\documentclass[11pt,reqno]{amsart}
\usepackage{amsmath}
\usepackage{amssymb}
\usepackage{mhequ}
\usepackage{color}
\usepackage{graphicx}
\usepackage{enumerate}
\usepackage[unicode=true,
 bookmarks=true,bookmarksnumbered=false,bookmarksopen=false,
 breaklinks=false,pdfborder={0 0 1},backref=false,colorlinks=true]{hyperref}

\newtheorem{maintheorem}{Theorem}

\newtheorem{theorem}{Theorem}[section] 
\newtheorem{lemma}[theorem]{Lemma}     
\newtheorem{corollary}[theorem]{Corollary}
\newtheorem{proposition}[theorem]{Proposition}

\newtheorem{definition}{Definition}

\newtheorem{remark}{Remark}

\newtheorem{example}{Example}

\newcommand{\R} {\mathbb R}

\newcommand{\supp}{\operatorname{supp}}

\newcommand{\rf}[1]{(\ref{#1})}

\begin{document}

\title[Representation of Markov chains by random maps]
 {Representation of Markov chains by random maps: existence and regularity conditions} 

\author{J\"{u}rgen Jost}
\address{J\"{u}rgen Jost\\
   Max-Planck-Institute for Mathematics in the Sciences\\
   Inselstr. 22\\
   04103 Leipzig\\
   Germany\\
and Department of Mathematics\\ University of Leipzig\\ 04081 Leipzig\\ Germany}
   \email{jost@mis.mpg.de}

\author{Martin Kell}
\address{Martin Kell\\
   Max-Planck-Institute for Mathematics in the Sciences\\
   Inselstr. 22\\
   04103 Leipzig\\
   Germany}
   \email{mkell@mis.mpg.de}

\author{Christian S. Rodrigues}
\address{Christian S.~Rodrigues\\
   Max-Planck-Institute for Mathematics in the Sciences\\
   Inselstr. 22\\
   04103 Leipzig\\
   Germany}
   \email{christian.rodrigues@mis.mpg.de}

\date{\today}

\begin{abstract}
  We systematically investigate the problem of representing Markov
  chains by families of random maps, and which regularity of these
  maps can be achieved depending on the properties of the probability
  measures. Our key idea is to use techniques from optimal transport
  to select optimal such maps. Optimal transport theory also tells us
  how convexity properties of the supports of the measures translate
  into regularity properties of the maps via Legendre
  transforms. Thus, from this scheme, we cannot only deduce the
  representation by measurable random maps, but we can also obtain
  conditions for the representation by continuous random
  maps. Finally, we present conditions for the representation of
  Markov chain by random diffeomorphisms.

\end{abstract}


\maketitle

\section{Introduction}

Amongst the main concerns of Dynamics, one usually wants to decide
whether asymptotic states of a given class of systems are robust under
small random fluctuations. Such randomness, corresponding to natural
fluctuations in physical processes, are represented by either a
\textit{Markov chain model} with localised transition or by a sequence
of \textit{random maps}. To see how they arise, consider a
discrete-time system $f$ from a given topological space $M$ into
itself. Suppose at each iteration of $f$ we allow a small mistake of
size, say, at most $\varepsilon >0$. Then a Markov chain is defined by
a family $\{p_{\varepsilon}(\, \cdot \,|x)\}$ of Borel probability
measures, such that every $p_{\varepsilon}(\, \cdot \,|x)$ is
supported inside the $\varepsilon$-neighbourhood of $f(x)$. The orbit
of our dynamics subject to such small errors is thus given by
sequences of random variables $\{x_{j}\}$, where each $x_{j+1}$ has
distribution $p_{\varepsilon}(\, \cdot \, |x_{j})$. Alternatively, one
could think of the orbit as being made by the iteration $x_{j} = g_{j}
\circ \cdots \circ g_{1}(x_{0})$, where each measurable $g_{j}$ is
picked at random $\varepsilon$-close, in a sense to be made more
precise, from the original map $f$. Endowing the collection of maps
$\{g_{j}\}$ with a probability distribution $\nu_{\varepsilon}$, we
say that the sequence of random maps is a \textit{representation of
  that Markov chain} if for every Borel subset $U$
\begin{equation}
\label{eq.repr-criteria}
p_{\varepsilon}(U|x) = \nu_{\varepsilon}(\{g : g(x) \in U \}).
\end{equation}

In fact, given any sequence of random maps, one can always find a
Markov chain which is represented by this scheme~\cite[D.4]{BDV05},
see also~\cite{ZaH07}; the Markov chain is simply given by
\rf{eq.repr-criteria}, and one only needs to check that this satisfies
the Markov chain criterion. The converse problem, however, is much
more subtle, as we shall see. This is exactly the subject of this
paper. In other words, we investigate under which conditions imposed
on the Markov chain one can obtain a representation by a random map
scheme and how its regularity properties are reflected.

The study of realisations of Markov chains via random maps goes back
to Blumenthal and Corson~\cite{BlC70}. They considered the case where
$M$ is a connected and locally connected compact metric space under
some strong requirements on the probability measures. Let us denote by
$\mathcal{P}(M)$ the space of all probability measures on $M$. In
addition, suppose each $x \mapsto p(\,\cdot\,|x)$, acting from $M$ to
$\mathcal{P}(M)$, is continuous relative to the weak* topology on
$\mathcal{P}(M)$. Then, if \textit{for each $x$ the support of
  $p(\,\cdot\,|x)$ is all of $M$}, they showed that it is possible to
obtain a probability measure $\nu$ on the space of continuous
transformations of $M$ such that condition (\ref{eq.repr-criteria}) is
fulfilled. Their proof is based upon the existence of a continuous
projection of the subset of $\mathcal{P}(M)$ whose support is all of $M$
onto the space of probability measures on an interval whose support is
the whole interval itself, and its continuous inverse. The assumption
of full support on the probability measures is essential to assure the
continuity of the maps. See for example~\cite{Kif86}.

Weakening this condition, Kifer showed that if $M$ is a Borel subset
of a complete separable metric space (Polish), then any Markov chain
on $M$ can be represented by a sequence of measurable random
maps~\cite{Kif86}. His idea was to use Borel measurable isomorphisms
of $M$ to Borel subsets of the unity interval, as it had previously
been shown by Kuratowski. Later, Quas \cite{Qua91} tackled the case
where $M$ is a smooth compact orientable Riemannian manifold. He
showed that probability families which are absolutely continuous with
respect to the normalised Riemannian measure whose density is smooth
can be represented by $C^{\infty}$-random maps.

Afterwards, Ara\'{u}jo~\cite{Ara00} showed how to construct families
of $C^{r}$-diffeo\-morphisms on the $n$-torus near an unperturbed
$C^{r}$-diffeomorphism. He took advantage of the parallelisability of
this manifold and of its quotient by integers. Then he used natural
projections to identify orthonormal vector fields from which he could
build these maps $C^{r}$-close to the original one; see~{\cite[Example
  1]{Ara00}}. Using a parametrised geodesic flow, he also showed the
existence of parametrised families of diffeomorphisms, around an
unperturbed one, of any compact boundaryless manifold; see~
{\cite[Example 2]{Ara00}}. Nevertheless, his procedures focus on
\textit{uniformly continuous perturbations}, requiring the small noise
to uniformly cover a ball of positive radius around the unperturbed
diffeomorphism. Furthermore, they do not yield a family of
diffeomorphisms from a given probability distribution.

More recently, Benedicks and Viana \cite[Example 1.7]{BeV06}, and
\cite[D.4]{BDV05} constructed random maps for small non-uniform
noise. They lift implicitly the measure to the tangent space at
$f(z)$, then try to transform the distribution to a fixed measure on
$[0,1]^{n}$. If this transformation is invertible, its inverse regular
enough, and varies smoothly with respect to $z$, then it is possible
to select a random continuous map representing the perturbation. For
topological reasons these constructions may fail on manifolds with
non-trivial tangent bundle. Namely, their constructions implicitly
assume the existence of global cross-section of the frame bundle.

In fact, it is not clear how to choose the random maps representing a
Markov chain, and there might be many possible such choices. In such a
situation, a basic strategy of geometric analysis is to select the
maps according to some optimisation principle. This typically has the
advantage that an object selected by an optimisation principle is not
just some solution of the problem at hand; it typically enjoys
additional properties derived from the optimisation, and these
properties can typically be usefully exploited. This is also the
strategy we adopt in the present paper. Since the maps should relate
different measures, it is natural to select them by optimising a
transportation problem between those measures. Thus, in this paper we
introduce techniques from optimal transport in order to tackle the
representation of Markov chains under different levels of regularity
of the maps. The paper is organised as follows. After presenting the
main definitions in Section~\ref{sec.definitions}, we review the main
ingredients from optimal transport theory, in
Section~\ref{sec.optimal.transp}, to be used in the remaining part of
the paper. In Section~\ref{sec.repr.meas}, we use optimal coupling to
prove Theorem~\ref{thm.main.meas.maps}, which shows how Markov chains
can be represented by measurable maps. Then, in
Section~\ref{sec.repr.cont}, we use Moser's coupling to show how
representation by continuous random maps arise; the content of
Theorem~\ref{thm.main.cont.repres}.  In the following
Section~\ref{sec.reg.dens}, we discuss the regularity of densities and
how they affect the properties of the transport maps.  Then in the
next Section~\ref{sec.lifting}, we use tangent bundle lifts of the
measures and certain transformations to a fixed measure to construct
continuous families of probabilities on the bundles. To tackle general
manifolds, we rely on the fact that the tangent bundle of a manifold
is always contained in a (smooth) trivial bundle, which can be seen by
taking an isometric (Nash) embedding
$M\hookrightarrow\mathbb{R}^{n}$. Then we lift (in a nice way) the
measures from the tangent bundle to this trivial bundle to get a
measure family $\{\mu_{x}\}_{x\in M}$ on $\mathbb{R}^{n}$. Using
optimal transport theory and its regularity theory we get
transformations to a fixed measure varying continuously with respect
to $x\in M$ so that we can select sections of this bundle varying
continuously, such that the distribution at a point $x$ represents the
measure $\tilde{\mu}_{x}$. (Smooth) projections to the tangent bundle
and the exponential map then give us the random continuous maps. These are
summarised in Theorem~\ref{thm.main.repr.diff}.  Assuming further
regularity, these maps are differentiable and we obtain random maps
$C^{1}$-close to $f$, and if the unperturbed map $f$ is a
diffeomorphism, we thus obtain a family of random
diffeomorphisms. Therefore, we give a geometric condition for the
representation of Markov chains by random diffeomorphisms.

In a subsequent paper, the methods developed here will be applied in
order to address stochastic stability of several classes of dynamical
systems. Starting from the seminal work of Kifer \cite{Kif86, Kif88},
we shall give conditions to stability in terms of Markov chains
without the \textit{a priori} assumptions of existence of random maps.

\section{Some notation and definitions}
\label{sec.definitions}

In this section we set up the notation and collect some main
definitions to be used throughout this paper. For a comprehensive
presentation on random perturbations of dynamics, see~\cite{Kif86,
  Kif88}. Although our main interest regards problems
  where the topological spaces under consideration are differential
  manifolds, some of the results that we will present are also true
  under lower requirements. We shall state it explicitly when that is
  the case.  When we consider an arbitrary manifold though, say $M$,
we suppose it to be compact and finite dimensional, equipped with some
Riemannian structure, fixed once and for all, which induces a distance
$d : M \times M \to \mathbb{R}$. We call $m$ its normalised Riemannian
volume form on $M$, i.e. $m(M)=1$, and unless otherwise stated, we
take absolute continuity with respect to $m$. As before, let us denote
by $\mathcal{P}(M)$ the space of all Borel probability measures on
$M$. As usual, $\mathcal{P}(M)$ is endowed with the weak* topology.
The gradient operator will be denoted by $\nabla$, and the divergent
by $\nabla \cdot$; the gradient of $f$ at the point $x$ will be
denoted by $\nabla_{x}f$ or $\nabla f(x)$; the Laplace operator,
\textit{i.e.}, the divergent of the gradient will be denoted by
$\Delta$, and we use the same notation for $\mathbb{R}^{n}$ and for
Riemannian manifolds.

Regarding measurability and continuity, we recall  Lusin's theorem
to be used in our proofs.
\begin{theorem}~\cite[Theorem 2.3.5]{Fed69}.
\label{thm.Lusin}
Let $M$ be a locally compact metric space, $\mu$ a Borel measure on
$M$, and $N$ a separable metric space. Let $f : M \to N$ be a
measurable map. Consider $A \subset M$ a measurable set with finite
measure. Then for each $\delta > 0$ there is a closed set $K \subset
A$, with $\mu(A \backslash K) < \delta$ such that the restriction of
$f$ to $K$ is continuous.
\end{theorem}

\subsection{Markov chains and random maps}
\label{sec.Markoc-ch}

Let $N$ be separable complete metric space. We shall consider families
of probability measures $(\mu_{x})_{x\in M}$ in $\mathcal{P}(N)$ given
by measurable maps $x\mapsto\mu_{x}$ with index set $M$. We speak of a
continuous family of probability measures if the maps $x \mapsto
\mu_{x}$ vary continuously from $M$ to $\mathcal{P}(N)$ relative to
the weak* topology. In many cases one has $N=M$ or
$N=\mathbb{R}^{n}$. Such families are sometimes called
\textit{continuous Markov kernels}.
\noindent Therefore, Markov chains are special Markov kernels obeying
some conditional probability with localised distribution.

Similarly, we can have a more general definition for our random maps.
For an auxiliary probability space $(\Omega, \mathcal{A},
\mathbb{P})$, consider a measurable collection of maps $\mathcal{F}:
\Omega \times M \to N$, $(\omega, x) \mapsto f_{\omega}(x)$. Then, we
call the family $(f_{\omega}: M \to N)_{\omega\in\Omega}$ random
measurable maps. If in addition each map in $(f_{\omega}: M \to
N)_{\omega\in\Omega}$ is continuous, or a diffeomorphism, then we say
that it is a family of \textit{random continuous maps}, or
\textit{random diffeomorphisms}, respectively. They are also known as
\textit{random fields}. A representation of $(\mu_{x})_{x\in M}$ is the
mapping $\mathcal{F}: \Omega \times M \to N$, $(\omega, x) \mapsto
f_{\omega}(x)$ such that for each $x$
\begin{equation}
\label{eq.representation}
\mu_{x} = \mathcal{F}_{*}\mathbb{P}.
\end{equation}
Thus, we say that $(f_{\omega})_{\omega\in\Omega}$ represents
$(\mu_{x})_{x\in M}$ if the distribution of $\omega\mapsto
f_{\omega}(x)$ equals $\mu_{x}$ for all $x\in M$.

\section{On optimal transport}
\label{sec.optimal.transp}

The remaining part of this paper is based upon techniques from
\textit{optimal transport}. Our main reference is the book by
Villani~\cite{Vil09}. For the sake of completeness, we sample and
collect in this section the concepts to be used along the
way. Readers familiar with optimal transport may wish to skip this
section and only refer back to it when needed. 

The basic problem in \textit{optimal transport}, as introduced by
Monge, consisted in moving a given distribution  like a pile of
sand from one place to another with a minimal \textit{cost}. There are
several possible ways to generalise and tackle this problem. For
example, the given mass to be transported can be thought of as a
distribution in an appropriate probability space. In other words,
given measurable spaces $M,N$, and probability measures $\mu$ in
$\mathcal{P}(M)$ and $\nu$ in $\mathcal{P}(N)$, we seek for a
\textit{coupling}, or a way to connect these two measures.  More
generally, one has the following definition.
\begin{definition}
  Let $(M, \mu)$ and $(N, \nu)$ be two probability spaces. We couple
  $\mu$ and $\nu$ by constructing two random variables $X,Y$ on some
  probability space $(\Omega, \mathbb{P})$, such that $law(X) = \mu$,
  $law(Y) = \nu$. The law or distribution of $(X,Y)$ is called
  coupling of $(\mu, \nu)$.
\end{definition}
\noindent In our context, $\mu$ and $\nu$ are the only laws we shall
be interested in, so we choose $\Omega = M \times N$.
There are several examples of couplings arising in different
contexts. 

The first generalisation of  Monge's original problem we can think of
is given in terms of \textit{transport maps}. That is, given measurable spaces
$M,N$, probability measures $\mu$ in $\mathcal{P}(M)$ and $\nu$ in
$\mathcal{P}(N)$, we seek for measurable maps $T : M \to N$, 
such that for all Borel $E \subset N$ one has $ \mu(T^{-1}(E)) =
\nu(E)$. This is an example of a so-called \textit{deterministic
  coupling}. The requirement of a transport map, however, is a strong
condition, and  this problem  may not have a solution
unless more restrictions are made. The canonical example is when $\mu$
is a Dirac measure and $\nu$ is not.

In order to avoid ill-posed problems, one alternatively should look
for weak solutions of the transport problem as it has been proposed by
Kantorovich.  In this case, we focus on probability measures $\gamma$
in $\mathcal{P} (M \times N)$, whose projections (or marginals) are
$\mu$ and $\nu$. In other words, let $\Gamma(\mu, \nu) \subset
\mathcal{P}(M \times N)$, such that the canonical projections
$\pi_{\mathcal{P}(M)} : \Gamma(\mu, \nu) \to \mathcal{P}(M)$ and
$\pi_{\mathcal{P}(N)} : \Gamma(\mu, \nu) \to \mathcal{P}(N)$
hold. Then the marginals are given by the
push-forward $\pi_{\mathcal{P}(M)*}\gamma = \mu$ and
$\pi_{\mathcal{P}(N)*}\gamma = \nu$. The \textit{Kantorovich
  minimisation problem} consists in obtaining
\begin{equation}
\label{eq.optimal-trans}
C(\mu, \nu) = \inf_{\gamma \in \Gamma(\mu, \nu)}\int_{M \times N}c(x, y)d\gamma(x, y),
\end{equation}
where, for a given \textit{cost} function $c : M \times N \to [0,
+\infty]$, the infimum runs over the joint probabilities $\gamma$ in
$\Gamma(\mu, \nu)$. The joint probability measures are called
\textit{transport plans}, the ones achieving the minimum,
\textit{optimal transport plans}, and $C(\mu, \nu)$ the optimal
transport cost. Thus, this coupling is called optimal transport
coupling. Obviously, the solution of the Kantorovich minimisation
problem depends on the choice of the cost
function. Although we state several of the auxiliary
  results from optimal transport in full generality, for our purpose
  we shall only use the quadradic cost function. The following theorem guarantees
the existence of optimal coupling.
\begin{theorem}[Existence of optimal coupling {\cite[Theorem 4.1]{Vil09}}]
  Given two Polish spaces $(M, \mu)$ and $(N, \nu)$, and a lower
  semicontinuous cost function bounded from below, then there always exist
  optimal couplings of $(\mu, \nu)$.
\end{theorem}

Notice that any transport map $T : M \to N$ induces a transfer plan
$\gamma$ defined by $(Id \times T)_{*}\mu$. In fact, we can also
impose conditions on the cost function such that the optimal coupling
is obtained by a deterministic coupling. The search for a
deterministic coupling (transport map) $T$ which minimises
Eq.~\ref{eq.optimal-trans} for a given cost function is called the 
\textit{Monge problem}. In other words, the Monge problem consists in finding
deterministic optimal couplings realising
\begin{displaymath}
\min \left \{ \int_{M}c(x, T(x))d\mu(x) : T_{*}\mu = \nu \right \},
\end{displaymath}
$c : M \times N \to [0, +\infty]$.
The following proposition ensures the
existence of a unique transport map solving the Monge problem.
\begin{proposition}[Solution of the Monge problem]
\label{prop.sol-monge-prob}
Let $M$ be a Riemannian manifold, $\mathcal{X}$ a closed subset of
$M$, with $dim(\partial\mathcal{X}) \leq n-1$ and $\mathcal{Y}$ an
arbitrary Polish space. Let $c : \mathcal{X} \times \mathcal{Y}
\to \mathbb{R}$ be a continuous cost function, bounded from below and
assume that for the
probability measures $\mu \in \mathcal{P}(\mathcal{X})$ and $\nu \in
\mathcal{P}(\mathcal{Y})$,  the optimal cost $C(\mu, \nu)$ is
finite. If the following conditions are fulfilled
\begin{enumerate}[i]
\item[i)] c is differentiable everywhere;
\item[ii)] $\mu$ is absolutely continuous;
\item[iii)] $\nabla_{x}c(x, \cdot)$ is injective where defined,
  i.e., if $x, y, y'$ are such that $\nabla_{x}c(x, y) = \nabla_{x}c(x,
  y')$, then $y = y'$,
\end{enumerate}
then there exists a unique (in law) optimal coupling $(x,y)$ of $(\mu,
\nu)$, and it is deterministic.
\end{proposition}
\begin{proof}
  The proof follows from Theorem 10.28, Proposition 10.7, and Remark
  10.33 of~\cite{Vil09}.
\end{proof}
\begin{corollary}

\label{cor.sol-monge-stab}
  Suppose that for each $k \in \mathbb{N}$ we have a sequence of
  continuous cost functions $c_{k} : \mathcal{X} \times \mathcal{Y}
  \to \mathbb{R}$ converging uniformly to $c : \mathcal{X} \times
  \mathcal{Y} \to \mathbb{R}$, where $c$ is defined as above. Let
  $(\nu_{k})_{k\in \mathbb{N}}$ be a sequence of probabilities on
  $\mathcal{Y}$ converging weakly to $\nu \in
  \mathcal{P(\mathcal{Y})}$, and assume that for each $k$ there exist
  measurable maps $T_{k} : \mathcal{X} \to \mathcal{Y}$, such that
  each $T_{k}$ is an optimal transport map between $\mu$ and
  $\nu_{k}$. Then $T_{k}$ converges to $T$ in probability, i.e.,
\begin{displaymath}
  \forall \varepsilon > 0 \quad \mu \left(\left\{x \in M; d\left(T_{k}(x), T(x)\right)\geq \varepsilon \right\} \right) \xrightarrow[k \to \infty]{} 0.
\end{displaymath}
\end{corollary}
\begin{proof}
  The proof follows from Proposition~\ref{prop.sol-monge-prob} above
  and~\cite[Corollary 5.23]{Vil09}.
\end{proof}
\begin{corollary}
\label{cor.sol-monge-square}
Let $M = \mathbb{R}^{n}$ and $c(x, y) = -x \cdot y$. Consider  two probability measures $\mu,
\nu$ on $M$, such that $\mu$ is absolutely
continuous, then the solution of Monge's problem can be written as
\begin{displaymath}
y = T(x) = x + \nabla \psi(x),
\end{displaymath}
where $\psi$ is some convex, lower semicontinuous function.
\end{corollary}
\begin{proof}
  The proof follows from Theorem 10.44, Particular case 10.45, and
  Particular case 5.3 of~\cite{Vil09}. See also
  Section~\ref{sec.reg.dens} below.
\end{proof}
\begin{remark}
  Proposition~\ref{prop.sol-monge-prob} is a slight variation of a more
  general theorem about the solution of the Monge problem. Its
  conditions can be weakened or replaced in a number of ways; see
  \cite[Theorem 10.28]{Vil09}.
\end{remark}

When the cost function is given in terms of distances in a metric
space, we can use (\ref{eq.optimal-trans}) in order to define a
distance between measures.
\begin{definition}[Wasserstein distances]
  Let $(M, d)$ be a Polish metric space, and $p \in [0, \infty)$. The
  Wasserstein distance of order $p$ between any two probability
  measures $\mu, \nu$ on $M$ is given by
\begin{equation}
\label{eq.wasserstein-dist}
W_{p}(\mu, \nu) = \left( \inf_{\gamma \in \Gamma(\mu, \nu)}\int_{M}d(x, y)^{p}d\gamma(x, y)\right)^{1/p}.
\end{equation}
\end{definition}
Using the Wasserstein distances we can define a space given by the
restriction on $\mathcal{P}(M) \times \mathcal{P}(M)$ on which $W_{p}$
takes finite values, or the space of probability measures with
\textit{finite moment of order p}.
\begin{definition}[Wasserstein spaces]
The Wasserstein space of order $p$ is given by 
\begin{equation}
\label{eq.wasserstein-space}
\mathcal{P}_{p}(M)= \left\{ \mu \in \mathcal{P}(M) : \int_{M}d(x_{0},x)^{p}\mu(dx)< \infty \right\}.
\end{equation}
The choice of $x_{0}$ is arbitrary and the space does not depend on this.
\end{definition}

As a last important example we shall mention the powerful Moser
coupling~\cite{Mos65, Vil09}. 
\begin{theorem}[Moser coupling]
  \label{thm.Moser-coupling}
  Consider a smooth compact Riemannian manifold $M$ and its volume
  form $m$. In addition, consider  H\"{o}lder
  continuous positive probability densities $\rho_{0}$ and $\rho_{1}$
  on $M$. Then there is a
  deterministic coupling of $\mu_{0} = \rho_{0} m$ and $\mu_{1} =
  \rho_{1} m$. In other words, there exists a measurable map $T$ such
  that for all Borel $E \subset M$, we have $\mu_{1}(E) =
  T_{*}\mu_{0}(E)$. Furthermore, if $\rho_{0}, \rho_{1}$ are $C^{k,
    \alpha}$ then $T$ is $C^{k +1, \alpha}$.
\end{theorem}
%
  The map $T$ is explicitly given, for each $x \in M$, in terms of a solution
  of the elliptic equation \[ \Delta u(x) = \rho_{0} - \rho_{1},\]
  where $\Delta$ denotes the Laplace operator. The transport map is
  obtained by defining the locally Lipschitz vector field
\[
\xi(t, x)=\frac{\nabla u(x)}{(1-t)\rho_{0}(x) + t\rho_{1}(x)},\]
which integrates to the flow $(T_{t}(x))_{0 \leq t \leq 1}$ with an
associated family of measures $(\mu_{t})_{0 < t < 1}$. In particular,
the time-1 map pushes forward $\mu_{0}$ to $\mu_{1}$.  See~\cite{Vil09}
for more details.


\section{Representation by measurable random maps}
\label{sec.repr.meas}

From this section we begin to apply the techniques from optimal
transport just presented. We shall start by tackling the problem of
representing a Markov chain by measurable continuous maps. Our result
implies Theorem 1.1 by Kifer~\cite[Ch. 1]{Kif86}. We also treat the
case of maps between different spaces, as introduced in
Section~\ref{sec.Markoc-ch}. The main result of this section is
\begin{maintheorem}
\label{thm.main.meas.maps}
Let $M$ be a locally compact metric space and $N$ a locally compact
Riemannian manifold. Consider a measurable family of probability
measures $(\mu_{x})_{x\in M}$ in $\mathcal{P}(N)$ with finite
$p$-moments for $p \geq 1$, i.e. for some $y\in N$ \[\sup_{x\in
  M}W_{p}(\delta_{y},\mu_{x})<\infty.\]
 
Then there exist separable random measurable maps $(f_{\omega}:M\to
N)_{\omega\in\Omega}$ representing $(\mu_{x})_{x\in M}$.
\end{maintheorem}
\begin{proof}
  Let $\nu$ be any probability measure absolutely continuous with
  respect to some volume measure on $N$, such that
  $W_{p}(\nu,\delta_{y})<\infty$. Then
  Proposition~\ref{prop.sol-monge-prob} shows that for each $x\in M$
  there is a unique optimal coupling realised by a (measurable)
  transport map $T_{x}:N\to N$, i.e., \[
  W_{p}(\nu,\mu_{x})^{p}=\int d(y,T_{x}(y))^{p}d\nu(y).\]
  Let $m$ be a Borel measure on $M$. By Lusin's theorem
  (Theorem~\ref{thm.Lusin}), given the family $(\mu_{x})_{x\in M}$ and
  $A \subset M$ with finite measure, for every $\delta>0$, there exists
  a set $K \subset A$, such that one has $m(A \backslash K) < \delta$
  and the restriction of the measurable family of probability $x \mapsto
  \mu_{x}$ is continuous on $K$. Take any sequence $x_{k} \to x$ with
  $\{x, x_{1}, x_{2},\ldots\} \subset K$. Then we have that
  $\mu_{x_{k}} \to \mu_{x}$, and Corollary~\ref{cor.sol-monge-stab}
  implies that $\nu\left(\{y \in N;~ d(T_{x_{k}}(y), T_{x}(y) \geq
    \varepsilon \}\right) \to 0$. Thus, these maps vary continuously
  in $x \in K$. Setting $(\Omega,\mathbb{P})=(\supp \nu,\nu)$, we can
  take for some random variable $Y:\Omega\to N$ with
  $\mbox{law}(Y)=\nu$ on $N$, and
  \[f_{\omega}(x):=T_{x}(Y(\omega)), ~x \in K, ~\omega \in \Omega.\]
  It follows that the maps satisfy (\ref{eq.representation}), thus
  they represent the Markov chain. Moreover, they are measurable in
  $\omega$ and continuous in $x \in K$. Therefore, they are separable
  and jointly measurable on $\Omega \times K$. The sets $K$ such that
  this property holds generate a $\sigma$-algebra. Thus, there is a
  unique extension to the completion of the generated $\sigma$-algebra
  of $\Omega \times A$, and therefore, the random maps are separable
  and jointly measurable on $\Omega \times A$.
\end{proof}


\section{Representation by continuous random maps}
\label{sec.repr.cont}

In this section we apply optimal transport and regularity theory to
give conditions for the representation of Markov chains by continuous
random maps. The main result of this section is the following theorem, which
is closely related to the main result of Quas~\cite{Qua91}. 

\begin{maintheorem}
\label{thm.main.cont.repres}
Let $M$ be a complete separable metric space and $N$ a compact
Riemannian manifold without boundary with normalised volume measure
$m$. Assume $(\mu_{x})_{x\in M}$ in $\mathcal{P}(N)$ to be a
continuous family of probability measures, where each $\mu_{x}$ is
absolutely continuous with respect to $m$ and has positive H\"{o}lder
continuous (for some exponent $\alpha >0$) probability density varying
continuously with $x \in M$ with respect to the $C^{0}$-topology.

Then $(\mu_{x})_{x\in M}$ can be represented by random continuous
maps $(f_{\omega})_{\omega\in\Omega}$.
\end{maintheorem}

The proof is based on an application of Moser's coupling. 
We start with the following result.
\begin{proposition}
  \label{prop.cont-repres}
   Let $M$ and $N$ be complete separable metric spaces and consider
  $(\mu_{x})_{x\in M}$ in $\mathcal{P}(N)$ a continuous family of
  probability measures. Suppose for a fixed measure
  $\nu\in\mathcal{P}(N)$ with compact support there exists a family of
  continuous maps $(T_{x}:\operatorname{supp}\nu\to N)_{x\in M}$
  varying continuously in the $C^{0}$-topology such that \[
  (T_{x})\nu=\mu_{x}.\] Then $(\mu_{x})_{x\in M}$ can be represented
  by random continuous maps $(f_{\omega})_{\omega\in\Omega}$, such
  that $(\Omega, \mathbb{P})=(\operatorname{supp}\nu,\nu)$.

  If, in addition, $T_{x}$ varies H\"{o}lder or Lipschitz continuously
  with respect
  to $x\in M$ then so does $f_{\omega}$ with the same constants
  (resp. exponents).
\end{proposition}
\begin{proof}
  Set $(\Omega,\mathbb{P})=(\operatorname{supp}\nu,\nu)$ and let
  $X:\Omega\to\Omega$ be any random variable such that
  $\operatorname{law}X=\nu$. Define
\[
f_{\omega}(x)=T_{x}(X(\omega)).\] By construction,
$(f_{\omega})_{\omega\in\Omega}$ represents $(\mu_{x})_{x\in M}$. We
need to show that $f_{\omega}:M\to N$ is continuous. Since $T_{x}$
varies continuously with respect to the $C^{0}$-topology and its
domain is compact we have for a fixed $T_{x}$ \[
d(T_{x}(X(\omega)),T_{y}(X(\omega))\le
d_{C^{0}}(T_{x},T_{y})<\epsilon\] whenever $d(x,y)<\delta$ for
sufficiently small $\delta=\delta(\epsilon,x)>0$.
\end{proof}
\begin{remark}
  \label{rmk.viana.exam}
  Proposition~\ref{prop.cont-repres} above generalises Example 1.7 by
  Benedicks-Viana~\cite{BeV06}, where $\nu$ is the Lebesgue measure on
  $[0,1]^{n}$ and $T_{x}$ is the inverse of a rearrangement $S_{x}$ of
  the positive measure
  $\mu_{x}\in\mathcal{P}([-\epsilon,\epsilon]^{n})$ to $[0,1]^{n}$,
  assuming $S_{x}$ is a homeomorphism which varies continuously in
  $C^{0}$.
\end{remark}
%
In order to obtain random continuous maps, it is actually enough to
assume that $x\mapsto T_{x}$ is pointwise continuous $\nu$-a.e. Then,
if $x\to y$ for $\nu$-a.e. $\omega$ we have $T_{x}(\omega)\to
T_{y}(\omega)$.
%
%
We remark also that the assumption on the continuity of each $T_{x}$
could be relaxed by using Lusin's theorem. Furthermore, using a
general version of the Kolmogorov-Chentsov continuity lemma for random
fields obtained in \cite{Pot09}, a similar result holds if we only
assume that each $T_{x}$ is Borel measurable converging fast enough as
$x\to y$ in the topology of convergence in probability with respect to
$(\Omega,\mathbb{P})$.

%

%
\begin{proof}[Proof of Theorem B]
  Let $\mu_{x}=\rho^{(x)}dm$. Then, by
    Theorem~\ref{thm.Moser-coupling}, there exists a coupling of
    $(m,\mu_{x})$ induced by the time-$1$ map of a Lipschitz
    continuous vector field varying continuously with respect to $x\in
    M$. It is given by the solution of the following elliptic equation
    on $N$\[ \Delta u^{(x)}=1-\rho^{(x)}.\] 
    Since $\rho^{(x)}$ is $C^{0,\alpha}$, Schauder's theorem, see for
    example~\cite[Ch. 11.2]{Jos06}, implies that $u^{(x)}$ is
    $C^{2,\alpha}$. So we define the vector field on $N$ \[
    \xi_{x}(t,y)=\frac{\nabla u^{(x)}(y)}{(1-t)+t\rho^{(x)}(y)}.\]
    According to our assumptions it is well-defined and integrates to
    a flow $T_{x,t}:N\to N$. Indeed, since $\rho$ is positive and
    $C^{0,\alpha}$, the vector field $\xi$ is likewise $C^{0,\alpha}$
    w.r.t. $y$. Therefore, the flow is $C^{1,\alpha}$
    w.r.t. $y$. Furthermore, we have \[(T_{x,1})_{*}m=\mu_{x}.\]
    Notice that since the densities vary continuously with respect to
    the $C^{0}$-topology, so does the vector fields. This continuous
    dependence gives rise to a family of $T_{x,1}$ also varying
    continuously with respect to the $C^{0}$-topology. Then,
    Proposition~\ref{prop.cont-repres} implies that $(\mu_{x})_{x\in
      M}$ can be represented by random continuous maps
    $(f_{\omega})_{\omega\in\Omega}$.
\end{proof}


\section{Regularity of densities}
\label{sec.reg.dens}

The next step is to establish conditions on representations of Markov
chains by random diffeomorphisms. Before doing that, we shall again
take up the discussion on optimal transport and its regularity
properties applied to the regularity of the densities of Markov
chains. We will focus on families of probabilities on
$\mathbb{R}^{n}$. After that, we show how to use the results of this
section on Riemannian manifolds via lifting and embedding
techniques. We start with some technicalities, showing how convexity
of the support of measures is related to the regularity of the
transport maps.

Consider lower semi-continuous functions $\phi : U \subset
\mathbb{R}^{n} \to \mathbb{R} \cup \{ +\infty \}$, and $\psi : U'
\subset \mathbb{R}^{n} \to \mathbb{R}$, which will be hereafter called
\textit{potentials}. Then, let us define at $y \in U' \subset
\mathbb{R}^{n}$
\begin{displaymath}
  \phi^{c}(y) = \sup_{x \in U}(-c(x,y) - \phi(x)).
\end{displaymath}
It is called the \textit{cost-transform} or \textit{c-transform} of $\phi$.
Furthermore, let us define
\begin{displaymath}
G_{\phi}(x)=\{y \in U' : \phi(x) + \phi^{c}(y) = -c(x,y)\}.
\end{displaymath}
Then one can prove the following general result.
\begin{theorem}
\label{thm.monge.potential}
  Let $U,U'\subset\mathbb{R}^{n}$ be bounded domains of
  $\mathbb{R}^{n}$, and $c(x,y) = -x \cdot y$. Let $\mu$, and $\nu$ be
  probability measures on $U$, and $U'$, respectively. Assume that $\mu$
  does not give mass to sets of Hausdorff dimension less than or equal
  to $n-1$. Then there exist a $\mu$-a.e. unique $T$ solving the
  Monge problem for this cost function $c$. Moreover, there is a convex potential $\phi$ on
  $U$, such that $T = G_{\phi}$. Finally, if $\psi$ is convex and
  satisfies $(G_{\psi})_{*}\mu=\nu$, then $\nabla \psi = \nabla \phi$,
  $\mu$-a.e.
\end{theorem}
\noindent
See, for example, \cite[Theorem 2.7]{Loe09}. The condition of not
giving mass to sets of Hausdorff dimension less than or equal to $n-1$
is satisfied if for some $p > n$, $\mu$ has $L^{p}$-density; see
Proposition 3.3 of \cite{Loe09}. In fact, for the cost function
$c(x,y) = -x \cdot y$, one can show that $\phi = \psi^{c}$, where
$\psi^{c}$ is the c-transform of $\psi$. In the case of this
particular cost function, they are Legendre transforms of each
other. See Particular case 5.3, and Definition 5.7 in~\cite{Vil09}.

Regarding the regularity of $\phi$, one can prove the following result
that we shall use in the sequel.

\begin{theorem}
\label{thm.loeper-reg}
Assume $c(x,y) = -x \cdot y$, and let $U,U'\subset\mathbb{R}^{n}$ be
bounded strictly convex. Suppose $\mu$ and
$\nu$ are probability measures on $U$ and $V^{'} \subset U'$,
respectively, with $V'$ being convex.
Assume that for the convex potential $\phi$ and the cost function
$c(x,y)$ we have $(G_{\phi})_{*}\mu=\nu$. Furthermore, denoting the
volume measure by $m$, assume that $\nu\ge \kappa m$ on $V^{'}$ for
some $\kappa >0$, and
$\mu$ satisfies for some $p\in]n,\infty]$ and $C_{\mu} > 0$,
\begin{equation}
\label{eq.integrability}
\mu(B_{\varepsilon}(x))\le C_{\mu}\varepsilon^{n(1-\frac{1}{p})} \quad \mbox{for all \ensuremath{\varepsilon>0} and \ensuremath{x \in U}}.
\end{equation}
%
Then $\phi$ is continuously differentiable on
$U_{\delta}=\{x\in U\,|\, d(x,\partial U)>\delta\}$ with
H\"{o}lder continuous derivatives for \textup{$\delta>0$}. In
particular, for some $\beta\in(0,1)$ and $\mathcal{C}$ depending only
on $U$,\textup{ $U'$,} $\kappa > 0$, $\delta>0$, $p$  and
$C_{\mu}$ \[
\|\phi\|_{C^{1,\beta}(U_{\delta})}\le\mathcal{C}.\]
If, furthermore, $\mu$ is supported on some $\bar{V}$ compactly
contained in $U$ then $G_{\phi}$ is a H\"{o}lder continuous map
with H\"{o}lder norm  bounded by $\mathcal{C}$.
\end{theorem}
\begin{proof}
  The proof directly follows from Theorem 3.4, Theorem 3.5 and
  Proposition 3.3 of \cite{Loe09} using the particular chosen cost
  function $c(x,y) = -x \cdot y$.
\end{proof}
\noindent We remark that the integrability condition
(\ref{eq.integrability}) is satisfied if, for example, for some $p >
n$, $\mu$ has $L^{p}$-density. See Proposition 3.3 of \cite{Loe09}.

Consider some compact $K \subset \mathbb{R}^{n}$, and assume
$(\mu_{x})_{x\in K}$ to be a family of measures, such that each
$\mu_{x}$ satisfies the assumption of $\mu$ in
Theorem~\ref{thm.loeper-reg}. Then, from
Theorem~\ref{thm.monge.potential} it follows that we can define
optimal transport maps $S_{x}$ between $\mu_{x}$ and $\nu =
m_{|[0,1]^{n}}$. Furthermore, from Theorem~\ref{thm.loeper-reg}, these
maps are H\"{o}lder continuous with H\"{o}lder norm bounded by a
constant only depending on some compact convex neighbourhood
$\Omega_{x}$ of the support of $\mu_{x}$. In particular, by uniqueness
of the optimal transport map, we have the following Lemma.
\begin{lemma}
\label{lemma.cont-fam}
  Let $(\mu_{x})_{x\in K}$ be a continuous family of probability
  measures on some compact set $K \subset \mathbb{R}^{n}$. Suppose
  that each $\mu_{x}$ satisfies the assumptions of
  Theorem~\ref{thm.loeper-reg} with $C=C_{\mu_{x}}$ independent of
  $\mu_{x}$ and all supports contained in some convex $U \subset
  \mathbb{R}^{n}$. If we assume that the supports of $\mu_{x}$ are
  contained in $D=\operatorname{cl}(U_{\delta})$ for some $\delta>0$
  then the optimal transport maps $S_{x}$ between $\mu_{x}$ and $\nu =
  m_{|[0,1]^{n}}$ vary uniformly, that is if $x_{n}\to x$ then
  $S_{x}:D\to U^{'}$ varies continuously in the uniform topology of
  $C^{0}(D,U^{'})$.
\end{lemma}
\begin{proof}
  Because the $\|\cdot\|_{C^{1,\beta}}$-norm of $\{\phi_{x}\}_{x\in
    K}$ is uniformly bounded, the set is pre-compact in
  $C^{1}(D)$. Since up to a constant the potentials are unique we can
  assume $\phi_{x}(y_{0})=0$ for all $x\in K$ and some $y_{0}\in
  D$. The limit of
  $\lim_{n\to\infty}\phi_{x_{n}}\to\tilde{\phi}_{x_{0}}$ for some
  $x_{n}\to x_{0}$ is also a convex potential solving the optimal
  transport problem with $\tilde{\phi}_{x_{0}}(y_{0})=0$, which implies
  $\tilde{\phi}_{x_{0}}=\phi_{x_{0}}$. Therefore $\{\phi_{x}\}_{x\in
    K}$ is already closed and thus compact and $S_{x}:D\to U'$
  varies continuously in the $C^{0}$-topology as $x$ varies in $K$.
\end{proof}
\begin{proposition}
  \label{prop.cts-inv}Let $(\mu_{x})_{x\in
    K}\subset\mathcal{P}(\mathbb{R}^{n})$ be as above. In addition,
  assume each $\mu_{x}$ is supported on a convex set and has strictly
  positive Lebesgue density on its support, i.e. $\mu_{x}\ge
  \kappa \cdot m$ for some $\kappa >0$ independent of $x$, then \[
  S_{x|\operatorname{supp}\mu_{x}}:\operatorname{supp}\mu_{x}\to[0,1]^{n}\]
  is continuously invertible. Furthermore, if the supports vary
  continuously with respect to the Hausdorff metric on (compact)
  subsets of $\mathbb{R}^{n}$ then the inverse
  $T_{x}=S_{x}^{-1}:[0,1]^{n}\to\operatorname{supp}\mu_{x}\subset U$
  varies continuously in the $C^{0}([0,1]^{n}, U)$.  In addition,
  such a family $(\mu_{x})_{x\in X}$ can be represented by random
  continuous maps.

\end{proposition}
\begin{remark}
  If we do not assume that the supports vary continuously then it is
  still possible to show that the maps $T_{x}$ converge pointwise on
  $[0,1]^{n}$ to $T_{y}$ as $x$ converges to $y$ in $K$.
\end{remark}
\begin{proof}
  By Lemma~\ref{lemma.cont-fam}, $S_{x}$ varies continuously in
  $C^{0}(D,\mathbb{R}^{n})$. Then Theorem~\ref{thm.monge.potential}
  implies that the optimal transport problem from $\nu$ to $\mu_{x}$
  has a unique continuous optimal transport map
  $T_{x}:[0,1]^{n}\to\operatorname{supp}\mu_{x}$ which, by
  Corollary~\ref{cor.sol-monge-square}, is the derivative of a
  differentiable convex potential $\psi_{x}$.  Because $\psi_{x}$ is
  (up to a constant) the Legendre transform of the potential
  $\phi_{x}$, whose derivative is $S_{x}$, we necessarily have
  $T_{x}=S_{x}^{-1}$ on $[0,1]^{n}$.

  To prove that $x\mapsto T_{x}$ is continuous it is sufficient to
  show that the graphs converge with respect to the Hausdorff
  metric. Because $x\mapsto S_{x}$ is continuous and the supports of
  $\mu_{x}$ vary continuously with respect to the Hausdorff metric,
  the restricted graph \[
  \widetilde{gr}(S_{x})=\{(y,S_{x}(y))\in\operatorname{supp}\mu_{x}\times[0,1]^{n}\}\]
  is continuous with respect to the Hausdorff metric on subsets of
  $\mathbb{R}^{n}\times\mathbb{R}^{n}$.  But this implies \[
  x \mapsto gr(T_{x})=(\widetilde{gr}(S_{x}))^{-1}\] is continuous and thus
  $x\mapsto T_{x}$ is continuous as well.

  In particular, Proposition~\ref{prop.cont-repres} implies that any
  such family $(\mu_{x})_{x\in K}$ can be represented by random
  continuous maps.

\end{proof}
\begin{proposition}
  Suppose that $\mu$ is absolutely continuous with respect to the
  Lebesgue measure $m$ with $L^p$-density for some $p>n$.  Furthermore,
  let $\nu=m_{|[0,1]^{n}}$. Then the transport map $S:\supp \mu \to
  [0,1]^{n}$ is continuously invertible on $U$, where $U$ is the set
  of point $z$ which admit a (convex) neighbourhood where the density
  of $\mu$ is strictly positive. In particular, this holds in
  $\{\rho>0\}=\operatorname{int}(\operatorname{supp}\mu)$ if the
  density $\rho$ of $\mu$ is continuous.
  \end{proposition}
\begin{proof}
By Theorem~\ref{thm.loeper-reg} the map $S$ is
    (H\"{o}lder) continuous.  Furthermore, the optimal transport map
    $T$ from $\nu$ to $\mu$ is almost everywhere the inverse of $S$ as
    they are the gradients of convex potentials $\phi$ and $\psi$,
    which are the Legendre transforms of each other. Thus, it suffices
    to show that $T$ is single-valued on $S(U)$. By definition we
  have
\begin{equation*}
S(x)=\{y\,|\,\phi(x)+\psi(y)=x\cdot y\}
\end{equation*}
and
\begin{equation*}
  \psi(y)=\sup\{xy-\phi(x)\},
\end{equation*}
which implies that $S^{-1}(S(x))$ is convex. So if we show that $T$ is
one-to-one on $S(V_{x})$ for some small neighbourhood $V_x$ for all
$x\in U$, then it implies that $S^{-1}(S(x))\cap V_{x}=x$, i.e. $S$
and $T$ are both single-valued and thus continuous on resp. $U$ and
$S(U)$.

Let $\pi$ be the optimal transport plan between $\nu$ and $\mu$, in
particular we have \[ (\operatorname{id}\times T)_{*}\nu=\pi.\]

Let $x$ be a point in $U$ and $V_{x}$ be a closed neighbourhood 
such that $\mu \ge \kappa m$ from some $\kappa > 0$.
By the restriction property for optimal
transport plans {\cite[Theorem 4.6]{Vil09}} the plan
$\tilde{\pi}_{|\mathbb{R}^{n}\times V_{x}}$ is an optimal transport
plan between its marginal, i.e. between some $\tilde{\nu}\le\nu$ and
$\tilde{\mu}=\mu_{|V_{x}}$. Obviously this plan is also induced by
$T$, i.e.\[ (\operatorname{id}\times T_{*})\tilde{\nu}=\tilde{\pi}.\] 
Also note that $\tilde{\nu}$ is supported on $S(V_{x})$.

The measures $\tilde{\nu}$ and $\tilde{\mu}$ satisfy the assumptions
of Theorem~\ref{thm.loeper-reg}, which implies that $T$ is
(H\"{o}lder) continuous on $S(V_x)$ and hence one-to-one.
\end{proof}
\begin{corollary}
\label{cor.cts-strictconv}
Let $\mu$ and $\nu$ be as above. Assume
the support of $\mu$ is strictly convex and \[
U=\{\rho>0\}=\operatorname{int}(\operatorname{supp}\mu)\]
where $\rho$ is the continuous density of $\mu$. Then the optimal
transport map $T$ from $\nu$ to $\mu$ is a continuous map from
$\operatorname{supp}\nu$ to $\operatorname{supp}\mu$.
\end{corollary}
\begin{proof}
  Restrict $S$ to $\operatorname{supp}\mu$. By the previous theorem
  $S$ is injective in the interior of its domain. Furthermore,
  $S^{-1}(S(x))$ is convex. Because the support of $\mu$ is strictly
  convex this also implies $S$ is injective on the boundary, i.e. if
  $x'\in S^{-1}(S(x))$ for $x\in\partial\operatorname{supp}\mu$ then
  $\lambda x'+(1-\lambda)x\in\partial\operatorname{supp}\mu$ which
  implies $x=x'$.

  Because $S$ is one-to-one on its (convex) domain
  $\operatorname{supp}\mu$ and
  $S(\operatorname{supp}\mu)=\operatorname{supp}\nu$, so is its
  inverse $T$ on $\operatorname{supp}\nu$. Thus, $T$ is continuous.
\end{proof}

The two previous results show that one can only control the behaviour
of the transport maps if the supports are strictly convex or the
support convex and the density positive everywhere.

\begin{proposition}
Let $K$ be some (compact) set and $B_{1}$ be the closed unit ball
in $\mathbb{R}^{n}$ and $(f_{x}:B_{1}\to f_{x}(B_{1})\subset\mathbb{R}^{n})_{x\in K}$
be a family of diffeomorphisms (onto their images) varying
continuously in $C^1$ 
w.r.t. $x$. Assume $(\mu_{x})_{x\in K}$ is a continuously varying
family of measures supported on the images of $f_{x}$, i.e.\[
\operatorname{supp}\mu_{x}=f_{x}(B_{1}).\]
If each $\mu_{x}$ has $L^{\infty}$ (resp. continuous)
density w.r.t. the Lebesgue measure then there is a (unique) continuously varying family $(\nu_{x})_{x\in K}$
of measures with $L^{\infty}$ (resp. continuous)
densities with $(f_x)_\ast \nu_x=\mu_x$.

Furthermore, if every point in the interior of the support of $\mu_{x}$
admits a neighbourhood such that the Lebesgue density of $\mu_{x}$
is strictly positive in the interior of the support then the same
holds for $\nu_{x}$. 
\end{proposition}
\begin{proof}
Assume $d\mu_{x}=\rho_{x}d m$ and there is some (Lebesgue regular)
family $d\nu_{x}=\vartheta_{x}d m$ such that $(f_{x})_{*}\nu_{x}=\mu_{x}$.
By the Jacobian equations we have \[
\vartheta_{x}(z)=\rho_{x}(f_{x}(z))\cdot J_{f_{x}}(z).\]
So defining $\vartheta_{x}(z)$ as above gives us $\nu_{x}$. 

Because $(f_{x})_{x\in K}$ are continuously varying diffeomorphisms
the Jacobians $J_{f_{x}}(z)$ vary continuously with respect to
$(x,z)\in K\times B_{1}$ which implies the statement of the
proposition.
\end{proof}
\begin{corollary}
\label{cor.regularity-cont-repres}
Let $(\mu_{x})_{x\in K}$ be as above. In addition, assume that each
$\mu_x$ has continuous density which is strictly positive on the 
interior of its support. Then there exists a family of continuous maps $T_{x}:[0,1]^{n}\to\mathbb{R}^{n}$ varying in
the $C^{0}$-topology and the following holds\[
(T_{x})_{*} m_{|[0,1]^{n}}=\mu_{x}.\]

In particular, all such families $(\mu_{x})_{x\in K}$ can be represented
by random continuous maps $(f_{\omega})_{\omega\in\Omega}$.
\end{corollary}
\begin{proof}
  Just note the previous proposition implies that the family $(\nu_x)_{x\in K}$
  and the measure $\nu=m_{|[0,1]^n}$ satisfies the assumptions of 
  Corollary~\ref{cor.cts-strictconv}. Therefore, we can apply 
  Proposition~\ref{prop.cts-inv} to get continuous maps 
  $\tilde{T}_x:[0,1]^n\to \mathbb{R}^n$ varying
  continuously in the $C^0$-topology such that $(\tilde{T}_{x})_{*} \nu=\nu_{x}.$
  Now it is easy to see that the maps $T_x = f_x \circ \tilde{T}_x$ satisfy 
  the required assumptions.
  Similarly one gets a random continuous map $(g_\omega)_{\omega\in\Omega}$ 
  representing $(\nu_v)_{x\in K}$. Then the random continuous map 
  $(f_x \circ g_\omega)_{\omega\in\Omega}$ is representing $(\mu_x)_{x\in K}$.
\end{proof}
\begin{remark}
The corollary can be applied if the supports are star-shaped with
differentiably varying centre and radial function, i.e. there are
$z_{x}\in B_{x}=\operatorname{supp}\mu_{x}$ and differentiable maps
\[
r_{x}:\mathbb{S}^{n-1}\to(0,\infty)\]
such that $x\mapsto(z_{x},r_{x})$ is continuous from $K$ to $\mathbb{R}^{n}\times C^{1}(\mathbb{S}^{n-1},(0,\infty))$
(continuity of $x\mapsto z_{x}$ is enough to show that $J_{f_{x}}(z)$
is continuous in $x$). 

The diffeomorphisms $f_{x}:B_{1}\to\mathbb{R}^{n}$ are constructed
via \[
f_{x}((\alpha,\rho))=z_{x}+(\alpha,r_{x}(\alpha)\cdot\rho).\]
(For simplicity we mixed polar coordinates $(\alpha,\rho)$ with Cartesian
$z_{x}$).\end{remark}


\section{Measures on bundles}
\label{sec.lifting}

In the previous section we addressed the regularity of random maps
depending on the properties of the family of measures. The
constructions were almost entirely carried out on $\mathbb{R}^{n}$.
In the next section, we shall tackle the problem of constructing
random diffeomorphisms on Riemannian manifolds. In order to use the
results on regularity, we shall lift the measures on the manifolds 
first to measures on their tangent bundles and 
then to measures on a trivial vector bundle containing the tangent
bundle. The idea is to use the local equivalence via the exponential
map of a neighbourhood of a point $p\in N$ and a neighbourhood of $0$
of the tangent space at that point~\cite{Jos11} 
and that via Nash's embedding theorem the tangent bundle is contained
in a trivial bundle. In this section we assume $M$ to be a complete separable 
metric space, and $N$ a locally compact Riemannian manifold. For $p \in N$, 
we shall locally endow $T_{p}N$ with the structure of a probability space via the
\textit{exponential map} $\operatorname{exp}_{p} : T_{p}N \to N$. We
say that $Q \subset T_{p}N$ is measurable if
$\operatorname{exp}_{p}(Q) \subset N$ is measurable.

Given $\mu_{x}\in\mathcal{P}(N)$, a measure $\tilde{\mu}_{x} \in
\mathcal{P}(T_{f(x)}N)$ can be defined as 
\begin{displaymath}
  \tilde{\mu}_{x} = (\operatorname{exp}^{-1}_{f(x)})_{*}\mu_{x},
\end{displaymath}
where $f: M \to N$ is some continuous map, and $f(x)$ belongs to the
support of $\mu_{x}$. For example, it might be its centre of mass. In
the case of randomly perturbed dynamics, in general $f$ is given by
the unperturbed system if one considers (bounded) random
perturbations.
Thus, for $x \in M$, the mapping
\begin{displaymath}
x\mapsto\mu_{x}\in\mathcal{P}(N),
\end{displaymath}
implicitly defines a mapping
\begin{displaymath}
x\mapsto \tilde{\mu}_{x}\in\mathcal{P}(T_{f(x)}N).
\end{displaymath}
Since the exponential map at a point $p\in N$ is a local
diffeomorphism between a neighbourhood of that point and a
neighbourhood of $0$ in the tangent space at $p$, we obtain the
following lemmata, whose proofs we leave to the reader.
\begin{lemma}
\label{lem.prob.bund}
Let $M$ be a complete separable metric space, $N$ a Riemannian
manifold. Consider a continuous family of probabilities
$(\mu_{x})_{x\in M}$ on $N$. Suppose that there is a $C^{r}$-map
$f:M\to N$, for $r \geq 0$, such that for each $x$, the support of
$\mu_{x}$ is contained in a sufficiently small neighbourhood
$U_{f(x)}$ of $f(x)$.

Then $(\mu_{x})_{x\in M}$ lifts to a continuous family of probability
measures $(\tilde{\mu}_{x})_{x\in M}$ on $TN$ with $\tilde{\mu}_{x}$
supported on $T_{f(x)}N$ (considered as a subset of $TN$).
\end{lemma}

\begin{lemma}
\label{lem.bundle.dens}
  If, in addition to the assumptions above, $\mu_{x}$ is absolutely
  continuous with respect to a volume form on $N$, then
  $\tilde{\mu}_{x}$ is absolutely continuous with respect to the
  Lebesgue measure on $T_{f(x)}N$ (and by equivalence to the standard
  Lebesgue measure on $\mathbb{R}^{n}$).

  Also, if $N$ is compact then the 
  the Lebesgue densities of $\mu_{x}$ and $\tilde{\mu}_{x}$ are
  comparable in the sense that they have the  same growth conditions, Lipschitz- or
  H\"older-constants or positivity properties of the density in the interior of
  their support. And, in particular, (strict) convexity of the support
  of $\mu_{x}$ is preserved if the support is contained in a ball around
  $f(x)$ with radius less than the convexity radius of $N$.
\end{lemma}

If the tangent bundle is parallelisable, i.e.
\begin{displaymath}
TN\cong N\times\mathbb{R}^{n},
\end{displaymath}
the mapping $x\mapsto\tilde{\mu}_{x}$ can be considered as a pair of
maps
\begin{displaymath}
x\mapsto(f(x),\tilde{\mu}_{x})\in N\times\mathbb{R}^{n}.
\end{displaymath}

More generally, if there is a trivial bundle $F\cong
N\times\mathbb{R}^{k}$ with a local equivalence $\rho : F \to N$ of a
neighbourhood of $0$ of $F_{f(x)}$ to the manifold (or the tangent
space at $T_{f(x)}N$), then we can lift the measures via the local
equivalence, by defining
\begin{displaymath}
\hat{\mu}_{x} = \rho_{*}\mu_{x}.
\end{displaymath}
Thus, for $x \in M$ we implicitly have the mapping
\begin{displaymath}
x\mapsto\hat{\mu}_{x}\in\mathcal{P}(F_{f(x)}),
\end{displaymath}
and therefore the family $(\hat{\mu}_{x})_{x\in M}$, which satisfies
the assumptions that the original family $(\mu_{x})_{x\in M}$ does.

In fact, one can always construct via isometric
  embedings a trivial bundle with a natural projection, and show that
the measures on the manifolds can be lifted to such a bundle. As it is
well know, Nash embedding theorem implies that the
tangent bundle $TN$ of a $n$-dimensional manifold $N$ is a sub-bundle
of the trivial bundle $F = N \times \mathbb{R}^{k}$ for some $k \ge n$
with a natural projection
\begin{displaymath}
\pi:F\to TN,
\end{displaymath}
which is a linear projection from $\mathbb{R}^{k}$ to $\mathbb{R}^{n}$
at each fibber at $p\in N$~\cite{Jos11}.

Suppose we have a fixed measure $\nu$ on $\mathbb{R}^{k}$ which has a
smooth density with respect to the Lebesgue measure on
$\mathbb{R}^{k}$, such that for any linear projection
$r:\mathbb{R}^{k}\to\mathbb{R}^{n}$ the measure \[
\nu^{(r)}=r_{*}\nu\] is absolutely continuous with respect to the
Lebesgue measure on $\mathbb{R}^{n}$ with positive density inside the
interior of its support. Then we claim that we can lift any Lebesgue
regular measure $\tilde{\mu}$ on $\mathbb{R}^{n}$ with support in the
interior of the support of $r_{*}\nu$ to a Lebesgue regular measure
$\hat{\mu}$ on $\mathbb{R}^{k}$. To show this, notice that
$\tilde{\mu}$ is absolutely continuous with respect to $\nu^{(r)}$,
i.e.
\begin{displaymath}
d\tilde{\mu}(x)=g(x)d\nu^{(r)}(x)
\end{displaymath}
then defining
\begin{displaymath}
d\hat{\mu}(y)=g(r(y))d\nu
\end{displaymath}
gives the required measure $\hat{\mu}$. Obviously, if the support of
$\tilde{\mu}$ is convex so is the support of $\hat{\mu}$, as well as
positivity in the interior is preserved by the lifts.

More generally, if there is a trivial bundle $F\cong N\times\mathbb{R}^{k}$
with a local equivalence $\rho:F\to N$ of a neighbourhood of $0$
of $F_{f(x)}$ to the manifold (or the tangent bundle at $T_{f(x)}N$)
such that we can lift the measures to
\begin{displaymath}
x\mapsto\hat{\mu}_{x}\in\mathcal{P}(F_{f(x)})
\end{displaymath}
with $(\hat{\mu}_{x})_{x\in M}$ satisfying the assumptions of the
previous section and
\begin{displaymath}
\rho_{*}\hat{\mu}_{x}=\mu_{x}
\end{displaymath}
then there is a family of random maps representing
$(\mu_{x})_{x\in M}$.


\section{Representation by random maps with higher regularity}

In this section, we want to address the representation of Markov
chains by random maps with higher regularity. In particular, our
ultimate goal is to give conditions for the representation by random
diffeomorphisms. Our first theorem in this section provides conditions
for the representation by continuous random maps which we shall use
thereafter to investigate the formal conditions for the representation
by random diffeomorphisms.

Before proceeding with, we present an example of the construction of
random diffeomorphisms for the case when $M = N$ is parallelisable,
based on~\cite[Example 1]{Ara00}.
\begin{example}
\label{ex.rand_maps_torus}
\emph{ Let $M$ be any parallelisable $n$-dimensional
  $C^{k}$-Riemannian manifold, for $k \geq 1$. Consider $f : M \to M$,
  a $C^{r}$-diffeomorphism for $1 \leq r \leq k$. We want to construct
  a family of random $C^{r}$-diffeomorphisms close to $f$ in the
  $C^{r}$-topology. Since $M$ is parallelisable, we have that $TM
  \cong M \times \R^{n}$, therefore there exists a globally
  orthonormal basis for the tangent space, that is, $n$ globally
  orthonormal vector fields $X_{1}(x) = (1, 0, \ldots, 0), \ldots,
  X_{n}(x) = (0, 0, \ldots, 1)$ in $\mathfrak{X}^{r}(M)$ for all $x
  \in M$. Consider the probability space $\Omega = [0,1]^{n}$. Then,
  for each $x \in M$, a family of random $C^{r}$-diffeomorphisms can
  be constructed as $f : \Omega \times M \to M$, where for $\omega =
  (\omega^{1}, \ldots, \omega^{n}) \in \Omega$ we have
\begin{displaymath}
  (\omega, x) \mapsto f(\omega, x) = f_{\omega}(x) := \exp_{f_{0}(x)}(\varepsilon\omega^{1}X_{1}(f_{0}(x)) + \cdots +
  \varepsilon\omega^{n}X_{n}(f_{0}(x))),
\end{displaymath}
where we set $f_{0}(x)
:= f(x)$ for every $x \in M$.
Notice that since $\omega$ is given by a certain distribution, not
necessarily the uniform distribution on $[0,1]^{n}$, it induces a
probability on the maps $f_{\omega}$. Taking $\varepsilon \to 0$
implies $||f_{\omega} - f_{0}||_{C^{r}} \to 0$.  Recall that the space
$\textrm{Diff}^{r}(M)$ of diffeomorphisms is open in $C^{r}(M)$, for
$r \geq 1$~\cite{Hir76}. Thus, the perturbation we describe induces a
probability on $\textrm{Diff}^{r}(M)$, giving the random
diffeomorphisms around $f_{0}$ with the same distribution of $\omega$
on the $C^{r}$-topology.}
\end{example}
\noindent The situation is more complicated however for more general manifolds
and perturbations. We start with the following result.

\begin{maintheorem}
\label{thm.main.repr.diff}
Let $M$ and $N$ be compact Riemannian $C^{k}$-manifolds without
boundary, with $k \geq 1$. Let $m$ be the normalised volume measure on
$N$. Consider $(\mu_{x})_{x\in M}$, a continuous family of probability
measures on $N$, such that each $\mu_{x}$ is absolutely continuous
with respect to $m$, with positive
continuous $L^{\infty}$-density $\rho_{x}$, and strictly convex
support. Suppose that there is a $C^{r}$-diffeomorphism $f:M\to N$,
for $r \leq k$, such that for each $x$, the support of $\mu_{x}$ is
contained in a sufficiently small neighbourhood $U_{f(x)}$ of
$f(x)$. Then $(\mu_{x})_{x\in M}$ can be represented by a family
$(f_{\omega})_{\omega\in\Omega}$ of $C^{r}$-random continuous maps.
\end{maintheorem}

The proof of this theorem is based on regularity theory, the result on
the conditions of representation of measures by continuous maps, and
the lifting properties of measures that we have presented in the
previous section. Indeed, as discussed in Section~\ref{sec.lifting},
it is possible to lift the measures on $N$ to its tangent bundle via
the exponential map and we can identify $TN$ with a sub-bundle of a
trivial bundle $N \times \mathbb{R}^{l}$~\cite{Jos11}. Then
Lemma~\ref{lem.prob.bund}, and Lemma~\ref{lem.bundle.dens} provide a
natural way of constructing a continuous family of probabilities on
the bundles according to $(\mu_{x})_{x\in M}$ on $N$. Furthermore, the
construction of a trivial bundle shows a natural way of embedding,
thus lifting again this measures to some $\mathbb{R}^{k}$, such that
the results of Section~\ref{sec.reg.dens} can be applied. In other
words, we can continuously select sections of the bundles. Thus, we
choose maps $C^{r}$-close to $f$ according to $\mu_{x}$.

\begin{proof}[Proof of Theorem \ref{thm.main.repr.diff}]
  We divide the proof into two cases. Namely, when the manifold $N$
  has trivial bundle, and when it has not.

\noindent\textit{Case I: Parallelisable manifolds:} Let us begin
with the case when $N$ has a trivial bundle. We start with the
following.

\noindent\textbf{Step 1}: \textit{There exists a continuous family $(\tilde{\mu}_{x})_{x \in M}$ on $TN$ and
  open neighbourhoods $V_{f(x)} \subset T_{f(x)}N$, such that every
  $\supp \tilde{\mu}_{x} \subset V_{f(x)}$.}  Indeed, since by
hypothesis the support of each $\mu_{x}$ is contained in some small
open $U_{f(x)}$, from Lemma~\ref{lem.prob.bund}, the family
$(\mu_{x})_{x \in M}$ is lifted to a continuous family
$(\tilde{\mu}_{x})_{x \in M}$ on $TN$. Furthermore, since the
probabilities are lifted via the exponential map, each
$\tilde{\mu}_{x}$ is supported in small neighbourhoods $V_{f(x)}
\subset T_{f(x)}N \subset TN$.

\noindent\textbf{Step 2}: \textit{Each $\tilde{\mu}_{x}$ has strictly convex support and is absolutely
  continuous with respect to the volume measure on $TN$, with
  densities $\gamma_{x}$ as regular as $\rho_{x}$.} Indeed,
Lemma~\ref{lem.bundle.dens} gives us the regularity conditions and
assures us the strictly positive density on the interior of the
support of each $\tilde{\mu}_{x}$. Furthermore, it shows that the
support of each $\tilde{\mu}_{x}$ is strictly convex.

\noindent\textbf{Step 3}: \textit{The family $(\tilde{\mu}_{x})_{x \in M}$ can be 
  represented by random continuous maps.} Note that $\tilde{\mu}_{x} =
\gamma_{x}Vol$, and that the density $\gamma_{x}$ can be written as
$\gamma_{x} = \exp_{f(x)}^{-1} \circ \rho_{x}$. Since $\rho_{x}$ is
$L^{\infty}(m)$, and the exponential map is locally a
$C^{\infty}$-diffeomorphism, we have that over bounded domains each
measure on the bundle also has bounded densities, i.e., $\gamma_{x}$
is $L^{\infty}(Vol)$. Furthermore, parallelisability of $TN$ implies that
$TN \cong N\times \mathbb{R}^n$. In particular, we have 
$\supp \tilde{\mu}_x \supset {f(x)}\times \mathbb{R}^n$. 
Therefore, we can assume that each $\tilde{\mu}_x$ lives on the same
$\mathbb{R}^n$. Now the continuous family of probability
that we have just constructed fulfils the conditions of
Corollary~\ref{cor.regularity-cont-repres}. Thus, they can be
represented by a random continuous map $(\tilde{f}_\omega:M\to TN)_{\omega\in\Omega}$.

\noindent\textbf{Step 4}: \textit{The
  family $(\mu_{x})_{x \in M}$ by random continuous maps.} 
  Using the exponential map $\exp:TN\to N$, set 
  $f_\omega := \exp \circ \tilde{f}_\omega$. 
  By construction of the measure family $(\tilde{\mu}_x)_{x\in M}$ one
  can easily verify that $(f_\omega)_{\omega\in\Omega}$ represents 
  $(\mu_x)_{x\in M}$.

\vspace{1cm}

\noindent\textit{Case II: General manifolds}
In this case note that $TN$ is contained in a trivial vector bundle 
$N\times\mathbb{R}^k$ such that the natural embedding 
$e:TN\to N\times\mathbb{R}^k$ is smooth.

\noindent\textbf{Step 1'} \textit{There exists a continuous family 
  $(\tilde{\mu}_{x})_{x \in M}$ on $N\times\mathbb{R}^k$ and
  open neighbourhoods $V_{f(x)} \subset f(x)\times\mathbb{R}^k$, such that every
  $\supp \tilde{\mu}_{x} \subset V_{f(x)}$.}
  Just note that as above one can first lift the measures $\mu_x$ to a family 
  $\hat{\mu}_x$ living on $TN$. By assumption, each measure has density w.r.t. 
  the Lebesgue measure on $T_{f(x)}N$. 
  As the embedding $e$ is smooth, if take the Lebesgue measure 
  on $\mathbb{R}^k \equiv {f(x)}\times\mathbb{R}^k$, restrict it 
  to a sufficiently small ball  $\lambda^k_{|B_\epsilon(0)}$, then the push-forward 
  has density w.r.t. the Lebesgue measure on $T_{f(x)}N$ which is smooth in 
  the interior of a small ball $B_\delta \subset T_{f(x)}N$.
  By duality, we can pull-back the densities of the measures $\hat{\mu}_x$ to 
  get measures $\tilde{\mu}_x$ with Lebesgue density.
  \noindent\textbf{Steps 2' and 3'} \textit{As steps 2 and 3 above.}

\noindent\textbf{Step 4'} \textit{The
  family $(\mu_{x})_{x \in M}$ by random continuous maps.}  As we
obtained a random continuous map $(\tilde{f}_\omega : M \to
N\times\mathbb{R}^k)$, we only need to define $f_\omega := e\circ \exp
\circ \tilde{f}_\omega$ to obtain the required family.

\end{proof}

\subsection{Conditions for a representation by random diffeomorphisms}

The proof of Theorem \ref{thm.main.repr.diff} showed that a crucial
step was the construction of random continuous maps $(f_\omega : M \to
\mathbb{R}^n)$. As the (sufficiently regular) solutions of the optimal
transport problem solve an elliptic partial differential equation,
namely the Monge-Amp\`ere equation, whose boundary conditions are
given in terms of densities, we can get regularity conditions from
general principles of elliptic regularity theory. Formally, this works
as follows.  Let $F:\Lambda \times C(N)\to \mathbb{R}$ be the solution
operator of an elliptic equation on $N$ depending on the parameter
$x\in \Lambda$, for some parameter space $\Lambda$. In our case,
$\Lambda$ will stand for boundary values, and in fact, we have
$\Lambda =M$. We then have
$$
0 = \frac{d}{d x} F = \frac{\partial F}{\partial x}  
                         + \frac{\partial F}{\partial \varphi} \cdot \frac{d\varphi}{dx}.
$$
In the elliptic case, for each $x$, the solution $\varphi$ is unique
and satisfies a-priori estimates, that is, it is controlled by the
data of the equation $F$. This means that we can control
$\frac{\partial \varphi}{\partial F}$. Therefore, the derivative
$\frac{\partial F}{\partial \varphi}$ is invertible, and we obtain
$$
\frac{d\varphi}{dx} = \left(\frac{\partial F}{\partial \varphi}\right)^{-1} 
                                     \frac{\partial F}{\partial x}.
$$
A similar formal calculation works for higher order derivatives. 

In our situation, the elliptic equation is the Monge-Amp\`ere
equation, i.e. $F(x,\varphi_x)=0$, if and only if, $\varphi_x$ solves
the Monge-Amp\`ere equation under the (second boundary) condition
$(\nabla \varphi_x)_* \mu_0 = \mu_x$, for which the regularity theory
is developed in \cite{MTW05,TW09}.

Thus, for our purposes, we only need to explicitly verify the
dependence of $F$ on $x$ and the ellipticity of the boundary value
problem, that is, a-priori estimates for a solution $\varphi$ of the
Monge-Amp\`ere equation under the boundary condition $(\nabla
\varphi)_* \mu_0 = \mu$. These have been obtained in \cite[Theorem
1.1]{TW09}. More precisely, for the case of a quadratic cost function
as considered here, that result yields a $C^2$-estimate for $\varphi$
in terms of the geometry of $N$, $\mu_0, \mu$ and $\sup |\varphi|$,
and this estimate implies uniqueness, see \cite[Theorem
1.2]{TW09}. From this, one may obtain $C^{2,\alpha}$-estimates, and
linear elliptic regularity theory then yields higher order estimates
in a standard manner.  Then the above formula yields the dependence of
the solution $\varphi_x$ on the parameter $x$. That is, a smooth
dependence of $\mu_x$ on $x\in M$ will translate into a corresponding
smooth dependence of $\varphi_x$ on $x$. In particular, the transport
maps will vary smoothly. In addition, we have the following proposition.
\begin{proposition}
\label{prop.ampere} 
  Assume that ellipticity conditions as discussed above are satisfied so that the solution of the Monge-Amp\`ere equation varies
  smoothly w.r.t. the parameter $x\in M$. If $f : M \to N$ is a
  diffeomorphism as in Theorem~\ref{thm.main.repr.diff} and the
  perturbation is sufficiently small, then the constructed random
  continuous map from Theorem~\ref{thm.main.repr.diff} is in fact a
  random diffeomorphism.
\end{proposition}

\begin{proof}
 Just note that the random map is constructed via an exchange of parameter, i.e.
 $$
 f_\omega (x) = T_x(\omega),
 $$
 where $\omega \in \Omega=B_1(0)\subset \mathbb{R}^n$.  The condition
 of smooth dependence on $x\in M$ is equivalent to saying that each
 $f_\omega$ is smooth (uniformly) dependent on $x \in M$. Therefore,
 if the perturbation is sufficiently small, then the maps converge on
 the $C^{r}$-norm to the unperturbed map $f$.  As the set of
 diffeomorphisms $\operatorname{Diff}(M,N)$ is open in $C^{r}(M,N)$,
 we have that $f_{\omega}$ must be a diffeomorphism as well.
\end{proof}

Therefore, Proposition~\ref{prop.ampere} and the discussion right
before it tell us that if $\mu_x$ fulfils the conditions in
Theorem~\ref{thm.main.repr.diff} and depends smoothly on $x\in M$, it
can be represented by random diffeomorphisms.

\section*{Acknowledgements}
The research leading to these results
has received funding from the European Research Council under the
European Union's Seventh Framework Programme (FP7/2007-2013) / ERC
grant agreement n$^\circ$~267087. M.K. was supported by the International Max Planck
Research School ``Mathematics in the Sciences''. We would like to thank the anonymous
referee for the careful reading and constructive comments, which have contributed to
substantially improve the presentation of this manuscript.
 

\begin{thebibliography}{amsalpha}

\bibitem[Ara00]{Ara00}
V. Ara\'{u}jo, \emph{Attractors and time averages for random maps.}, Annales de l'Institut Henri Poincar\'{e} Ð Analyse
Non Lin\'{e}aire \textbf{17} (2000), no.~3, 307--369.

\bibitem[BeV06]{BeV06}
M. Benedicks and M. Viana, \emph{Random perturbations and statistical properties of H\'{e}non-like maps.}, Annales de l'Institut Henri Poincar\'{e} Ð Analyse
Non Lin\'{e}aire \textbf{23} (2006), no.~5, 713--752.

\bibitem[BlC70]{BlC70}
R. M. Blumenthal and H. H. Corson, \emph{On continuous collections of measures.}, Annales de l'Institut Fourier  \textbf{20} (1970), no.~2, 193--199.

\bibitem[BDV05]{BDV05}
C. Bonatti, L. J. D\'{i}az, and M. Viana, \emph{Dynamics beyond uniform hyperbolicity: a global geometric and probabilistic perspective}, Encyclopedia of {M}athematical {S}ciences, vol 102, Springer, Berlin, 2005.

\bibitem[Fed69]{Fed69}
H. Federer, \emph{Geometric measure theory}, Grundlehren der Mathematischen Wissenschaften, vol. 153, Springer, New York, 1969.

\bibitem[Hir76]{Hir76} M. W. Hirsch, \emph{Differential topology}, Springer-Verlag, Berlin, 1976.

\bibitem[Jos06]{Jos06}
J. Jost, \emph{Partial Differential Equations}, Springer, Berlin, 2006.

\bibitem[Jos11]{Jos11}
J. Jost, \emph{Riemannian Geometry and Geometric Analysis}, Springer, Berlin, 2011.

\bibitem[Kif86]{Kif86}
Yu. Kifer, \emph{Ergodic theory of random transformations.}, Birkh\"{a}user, Boston, 1986.

\bibitem[Kif88]{Kif88}
Yu. Kifer, \emph{Random perturbations of dynamical systems.}, Birkh\"{a}user, Boston, 1988.

\bibitem[Loe09]{Loe09}
G. Loeper, \emph{On the regularity of solutions of optimal
  transportation problems}, Acta Mathematica \textbf{202} (2009),
no.~2, 241--283.


\bibitem[MTW05]{MTW05} X.N.Ma, N.Trudinger and X-J.Wang, \emph{Regularity of potential functions of the optimal transportation
problem}, Arch. Rat. Mech. Anal.\textbf{177}(2005), 151-183.


\bibitem[Mos65]{Mos65}
J. Moser, \emph{On the volume elements on a manifold}, Trans. Amer. Math. Soc., \textbf{120} (1965), 286--294.


\bibitem[Pot09]{Pot09}
J. Potthoff, \emph{Sample properties o random fields {II}: continuity}, Communications in stochastic analysis, \textbf{3} (2009), no.~3, 331--348.


\bibitem[Qua91]{Qua91}
A. N. Quas, \emph{On representation of {M}arkov chains by random smooth maps}, Bulletin of the {L}ondon {M}athematical {S}ociety \textbf{23} (1991), no.~5, 487--492.


\bibitem[TW09]{TW09}
N.Trudinger and X.J.Wang, 
\emph{On the second boundary value problem for Monge–Amp\`ere type equations and optimal
transportation}, Ann. Scuola Norm. Sup. Pisa Cl. Sci.\textbf{ 8} (2009), 143–174.


\bibitem[Vil09]{Vil09}
C. Villani, \emph{Optimal transport: old and new}, Grundlehren der Mathematischen Wissenschaften, vol. 338, Springer, Berlin, 2009.

\bibitem[ZaH07]{ZaH07}
H.~Zmarrou and A.J. Homburg, \emph{Bifurcations of stationary measures of random diffeomorphisms}, Ergodic Theory and Dynamical Systems \textbf{27} (2007), no.~5, 1651--1692.


\end{thebibliography}



\bibliographystyle{amsalpha}

\end{document}